\documentclass[11pt]{amsart}
\usepackage{amsfonts}
\usepackage{amsmath,graphicx,amssymb,amsthm}
\usepackage{amsmath}
\usepackage{amssymb}
\usepackage{graphicx}
\usepackage{comment}
\usepackage[backref]{hyperref}
\providecommand{\U}[1]{\protect\rule{.1in}{.1in}}

\usepackage[all]{xypic} 
\usepackage{tikz-cd}

\newcommand{\sett}[1]{\left\{#1\right\}}
\newcommand{\Real}{\mathbb R}

\newcommand{\Hil}{\mathcal{H}}

\newcommand{\E}{\mathbb{E}}

\newcommand{\ip}[2]{\left< #1, #2  \right>}

\def\U{\mathcal U}

\def\amslatex{$\mathcal{A}\kern-.1667em\lower.5ex\hbox{$M$}\kern-.125em\mathcal{S}$-\LaTeX}

\newcommand{\norm}[1]{\left\Vert#1\right\Vert}
\newcommand{\h}{\mathcal{H}}

\newtheorem{set}{set}[section]

\newcommand{\enifed}{\mathrel{\hbox{$\equiv$\hskip -.90em \lower .47ex \hbox{$\rightharpoondown$}}}}

\newtheorem{theorem}[set]{Theorem}
\theoremstyle{plain}

\newtheorem{corollary}[set]{Corollary}

\newtheorem{definition}[set]{Definition}

\newtheorem{lemma}[set]{Lemma}

\newtheorem{proposition}[set]{Proposition}

\allowdisplaybreaks[4]

\newcommand{\Cplx}{\mathbb C}

\begin{document}
\title[Noncommutative Joinings II]{Noncommutative Joinings II}
\author{Jon P. Bannon}
\thanks{JB is partially supported by a Lancaster University (STEM) Fulbright Scholar Award}
\address{Siena College Department of Mathematics, 515 Loudon Road, Loudonville,
NY\ 12211, USA}
\author{Jan Cameron}
\address{Department of Mathematics and Statistics, Vassar College, Poughkeepsie, NY 12604, USA}
\thanks{JC partially supported by Simons Foundation Collaboration Grant for Mathematicians \#319001.}
\author{Kunal Mukherjee}
\address{Indian Institute of Technology Madras, Chennai 600 036, India}
\thanks{KM partially supported by Vassar College Rogol Distinguished Visitor Program.}
\keywords{von Neumann Algebras, Ergodic Theory, Joinings}

\begin{abstract}
This paper is a continuation of the authors' previous work on noncommutative joinings, and contains a study of relative independence of W$^*$-dynamical systems. We prove that, given any separable locally compact group $G$, an ergodic W$^{*}$--dynamical $G$--system $\mathfrak{M}$ with compact subsystem $\mathfrak{N}$ is disjoint relative to $\mathfrak{N}$ from its maximal compact subsystem $\mathfrak{M}_{K}$ if and only if $\mathfrak{N}\cong\mathfrak{M}_{K}$. This generalizes recent work of Duvenhage, which established the result for $G$ abelian.
\end{abstract}
\maketitle


\section{Introduction}

This paper is a continuation of \cite{BCM}, in which we studied the basic analytical properties of joinings of W$^{*}$--dynamical systems, and used the theory of joinings to establish noncommutative analogues of a number of fundamental characterizations of ergodicity and mixing properties from classical ergodic theory.


Recall that a W$^{\ast}$-dynamical system is a tuple $\mathfrak{N}=(N,\rho,\alpha, G )$, where $N$ is a von Neumann algebra, $\rho$ is a faithful
normal state on $N$, and $\alpha$ is a strongly continuous action of a separable
locally compact group $G$ on $N$ by $\rho$-preserving automorphisms.  A joining of two W$^*$-dynamical systems expresses a relationship between the two systems that generalizes the situation in which they contain a common subsystem. Recall the following definition, which was motivated by its compatibility with mixing properties observed in the noncommutative setting. 
\begin{definition}
\label{DefinitionJoiningasState}A \textbf{joining} of the $W^{\ast}
$--dynamical systems $\mathfrak{N}=(N,\rho,\alpha, G )$ and $\mathfrak{M}
=(M,\varphi,\beta, G )$ is a state $\omega$ on
the algebraic tensor product $N\odot M^{op}$ satisfying
\begin{align}
\omega(x\otimes1_{M}^{op})  &  =\rho(x),\label{Eq: Pushforwards}\\
\omega(1_{N}\otimes y^{op})  &  =\varphi(y),\nonumber
\end{align}
and%
\begin{align}
\omega\circ(\alpha_{g}\otimes\beta_{g}^{op})  &  =\omega,
\label{Eq: Diagonal Action Preservation}\\
\omega\circ(\sigma_{t}^{\rho}\otimes(\sigma_{t}^{\varphi})^{op})  &
=\omega,\nonumber
\end{align}
for all $x\in N$, $y\in M$, $g\in G $ and $t\in\mathbb{R}$. We denote by
$J_{S}(\mathfrak{N},\mathfrak{M})$ the set of all joinings of
$\mathfrak{N}$ and $\mathfrak{M}$.
\end{definition}

It was shown in \cite{BCM} that a joining of systems $\mathfrak{N} = (N,\rho, \alpha, G)$ and $\mathfrak{M} = (M, \varphi, \beta, G)$ can, alternatively, be viewed as a pointed correspondence $\Hil$ between $N$ and $M$ with some additional structure, and as an equivariant Markov map $\Phi: N \rightarrow M$ (cf. \cite{BCM} Theorem 4.6).  In the present work, we study the special case in which such a Markov map is bimodular with respect to an embedded subalgebra of $N$ arising from a common subsystem $\mathfrak{B}$ of $\mathfrak{M}$ and $\mathfrak{N}$.  These maps are called joinings of $\mathfrak{N}$ and $\mathfrak{M}$ relative to (or, simply, ``over") the subsystem $\mathfrak{B}$ and have been employed widely in the study of classical, measure-preserving systems.   For instance, the celebrated Furstenburg Multiple Recurrence Theorem is a consequence of the Furstenberg-Zimmer Structure Theorem for ergodic measure-preserving dynamical systems, one proof of which \cite{Gl} showcases the joinings of an ergodic system with itself relative to a subsystem.

The question of whether a theory of joinings relative to subsystems can be formulated for non-abelian acting groups dates back at least to \cite{GlWe}, in which two proofs were given that the extension of a zero topological entropy transformation $T:X\rightarrow X$ to the space of probability measures M(X) (with the w$^{*}$-topology) must also have zero topological entropy. The first proof is joining-theoretic and the second combinatorial. The combinatorial proof generalizes to extend this result to amenable acting groups, and the authors suggested that this generalization may also admit a joining-theoretic proof. Though we do not pursue that specific question, in this paper we build a theoretical framework for approaching questions of this type -- our results hold for W*-dynamical systems arising from general locally groups acting on general von Neumann algebras with separable preduals.

A number of technical challenges arise in this setting.  First, as noted in \cite{Du1,Du2,BCM}, the dynamics of ergodic systems with non-abelian acting groups is somewhat more complicated.  For instance, such groups may admit irreducible finite-dimensional (hence, non-weakly mixing) representations, so may give rise to compact systems in which there are no eigenoperators.  Thus, to extend the results of \cite{Du2} to this new setting, different methods are required; these are set out in Section \ref{section:relative ergodicity and mixing properties}, which also draws on the previous work of the authors in \cite{BCM} and \cite{BCM2}.

A second challenge is how to extract information from a canonical projection associated to a joining of two systems. In the classical setting, a joining $\lambda$ of two p.m.p. $G$-systems $\mathcal{X}=(X,\mu, \alpha)$ and $\mathcal{Y}=(Y,\nu,\beta)$ is encoded in an equivariant Markov map, namely the restriction to $L^{2}(X,\mu)$ of the projection $L^{2}(X\times Y,\lambda) \rightarrow L^{2}(Y,\nu)$. This map restricts further to an equivariant map of $L^{\infty}(X,\mu)$ into $L^{\infty}(Y,\nu)$. However, the naive noncommutative analogue of this projection will not, in general, restrict to a map between von Neumann algebras. The precise relationship, articulated in Lemma \ref{ModularSymmetry} below, between this ``classical'' projection with the Markov map associated to a noncommutative joining as in Theorem 4.1 of \cite{BCM}, is essential for our primary description of noncommutative joinings over a common subsystem (Proposition \ref{States vs UCP}) and subsequent characterization (Theorem \ref{Partial Isometry}) of the notion of independence relative to a common subsystem.
The proof depends ultimately on the technical results of \cite{BCM2}, in which it was shown that the canonical Hilbert space operator associated to a Markov map between two systems exhibits a useful symmetry with respect to the modular operators and automorphism groups of the two systems.  

The contents of the paper are organized as follows. Section \ref{section:joinings over a common subsystem} recalls some of the notation and basic structural results of \cite{BCM}, and employs the modular symmetry results from \cite{BCM2} to obtain a useful intertwining lemma for Markov maps $\Phi:(N,\rho) \rightarrow (M,\phi)$ between von Neumann algebras, a consequence of which is Theorem \ref{Projection Formula}, which yields a formula -- valid on the dense subalgebra of analytic elements -- for the above mentioned Hilbert space projection associated to a joining of two systems.  We use these observations to formulate the notion of a joining relative to a common subsystem in the language of correspondences, and Proposition \ref{States vs UCP} connects this point of view to the one developed in \cite{Du2} which emphasizes states.  This yields a  formulation of disjointness of two W$^{*}$-dynamical systems relative to a common subsystem in the language of Markov maps and corresponedences.  

Section \ref{section:relative independence and orthogonality}  lays out some technical tools needed for our main results in the following section.  Of central importance is the fact that a ucp map $\Phi: N \rightarrow M$ is a joining of W$^*$ systems $\mathfrak{N}=(N,\rho,\alpha,G)$ and $\mathfrak{M}=(M,\varphi,\beta,G)$ relative to a subsystem $\mathfrak{B}$ if and only if its adjoint ucp map $\Phi^*$ is also a joining of the two systems relative to the same system.  This observation leads to a characterization of the relatively independent joining in terms of a ``projection property" of the associated ucp map in Theorem \ref{Partial Isometry}; and a characterization of the relatively independent joining in terms of the associated orthogonal projection in Theorem \ref{relative independence characterization}.

We turn in  Section \ref{section:relative ergodicity and mixing properties} to the dynamics of ergodic systems, and our main results of the paper.  We introduce the Kronecker subsystem of an ergodic W$^*$-dynamical system $\mathfrak{N}=(N,\rho,\alpha,G)$, whose underlying von Neumann algebra is generated by the finite-dimensional subspaces of $N$ which are invariant under the action $\alpha$.  Theorem \ref{theorem:Kronecker} establishes that these subspaces do in fact generate a compact subsystem $\mathfrak{N}_K$ of $\mathfrak{N}$, whose underlying von Neumann algebra is tracial and injective, and which turns out to be maximal among compact subsystems of $\mathfrak{N}.$  The notion of a maximal compact subsystem of a W$^*$-dynamical system was studied in \cite{Du2}, under the assumption that the acting group is abelian, but the methods used there required an abelian acting group, so a different approach is necessary here. A central role is played by the main results of \cite{BCM}, which control the finite-dimensional invariant subspaces of the unitary representation associated to an ergodic group action.  Our main result in Section \ref{section:relative ergodicity and mixing properties} extends Theorem 5.6 of \cite{Du2} from the setting of abelian groups to that of a general locally compact acting group, and combines techniques from \cite{BCM} and the technical methods introduced in the preceding sections to give a characterization of the Kronecker subsystem of an ergodic system in terms of relative disjointness.

\section{Joinings over a common subsystem}
\label{section:joinings over a common subsystem}
In this section, we develop the notion of a joining of two noncommutative dynamical systems over a common subsystem, from the point of view of correspondences.  An essential ingredient will be the main technical result of \cite{BCM2}, in which it was observed that the ucp map associated to a joining exhibits a useful symmetry with respect to the modular automorphism groups of the two underlying dynamical systems.

\subsection{Modular Symmetry of Joinings and Associated Projections}
If $N$ and $M$ are von Neumann algebras with normal, faithful states $\rho:N \rightarrow \mathbb{C}$ and $\varphi:M \rightarrow \mathbb{C}$, a linear map $\Phi: N \rightarrow M$ is called a $(\rho,\varphi)$--Markov map if $\Phi$ is unital, completely positive, and satisfies 
\[\varphi \circ \Phi = \rho, \quad \text{ and } \quad \sigma_t^\varphi \circ \Phi = \Phi \circ \sigma_t^\rho, \text{ for all }  t \in \mathbb{R}.\]If $\omega$ is a joining of $W^{\ast}
$--dynamical systems $\mathfrak{N}$ and $\mathfrak{M}$ as above, let $\mathcal{H}_{\omega}$ denote its associated GNS Hilbert space, i.e. the separation and completion of $N\odot
M^{op}$ with respect to the inner product extending
$\langle x_{0}\otimes y_{0}^{op},x_{1}\otimes y_{1}^{op}\rangle_{\omega}:=\omega
(x_{1}^{\ast}x_{0}\otimes(y_{0}y_{1}^{\ast})^{op})$ for all $x_{0},x_{1}\in N$
and $y_{0},y_{1}\in M$. The GNS representation $\pi_{\omega}$ of $N\odot
M^{op}$ will be viewed as defining a bimodule structure by $x\eta y:=\pi_{\omega}(x\otimes y^{op})\eta$ for all $x\in N$, $y\in M$ and $\eta\in\mathcal{H}_{\omega}$. We denote by $\xi_{\omega}$ the canonical cyclic vector associated to the class of $1_{N} \otimes 1_{M^{op}}$ in $\mathcal{H}_{\omega}$.   Recall from Theorem 3.3 of \cite{BCM} that, if we denote by $\Omega_\rho$ the canonical cyclic and separating vector for $L^2(N,\rho)$, then the map $  x \Omega_\rho  \mapsto  x \xi_\omega $ extends to an operator $L=L_{\Omega_{\rho},\xi_{\omega}}: L^2(N,\rho) \rightarrow \h_\omega$ called the \emph{left exchange} map that is a unitary left $N$--module isomorphism from $L^2(N,\rho)$ onto $\overline{N\xi_{\omega}}$.  Similarly, the \emph{right exchange} map $R = R_{\Omega_{\varphi},\xi_\omega}: L^2(M,\varphi) \rightarrow \h_\omega$ extends $\Omega_\varphi y \mapsto \xi_\omega y$ to a unitary right $M$--module isomorphism of $L^{2}(M,\varphi)$ onto $\overline{\xi_{\omega}M}$. In \cite{BCM} it was also shown that \[  \Phi_\omega(x) = R^\ast \pi_N(x) R\] defines a normal $(\rho,\varphi)$--Markov map $\Phi_\omega: N \rightarrow M$. Let $T_{\Phi_{\omega}}$ be the extension of the map $x\Omega_{\rho}\mapsto \Phi_{\omega}(x)\Omega_{\varphi}$ to $L^{2}(N,\rho)$, and denote by $P_{\omega}$  the map which is the restriction to the subspace $\overline{N\xi_{\omega}}$ of the orthogonal projection of $\mathcal{H}_{\omega}$ onto $\overline{\xi_{\omega}M}$. This is a natural analogue of the (classically) Markov projection mentioned in the introduction. The left and right exchange maps witness a relationship between $T_{\Phi_\omega}$ and $P_\omega$, as follows:

\begin{lemma}\label{MainTheorem}
Given a joining $\omega$ of the systems $\mathfrak{N}$ and $\mathfrak{M}$, we have
\begin{align}\label{Eq: Main Relation}
R_{\Omega_{\varphi},\xi_\omega}T_{\Phi_{\omega}}=P_{\omega}L_{\Omega_{\rho},\xi_{\omega}}.
\end{align}

\end{lemma}

\begin{proof}  For any $x \in N$ and $y \in M,$ one has

\begin{align}
\langle T_{\Phi_{\omega}}(x\Omega_\rho), \Omega_{\varphi} y \rangle_{\varphi}&=\langle R^{\ast}\pi_{N}(x)R\Omega_{\varphi}, \Omega_{\varphi} y \rangle_{\varphi} \\\nonumber
&=\langle x\xi_{\omega}, \xi_{\omega} y\rangle_{\omega}\\ \nonumber
&=\langle P_{\omega} x\xi_{\omega}, \xi_{\omega} y \rangle_{\omega}\\\nonumber
&=\langle R^{\ast} P_{\omega}L (x\Omega_{\rho}), \Omega_{\varphi} y\rangle_{\varphi}.
\end{align}
By the boundedness of the exchange maps, and of $T_{\Phi_\omega}$ and $P_\omega$, it follows that $T_{\Phi_\omega}$ and $R^* P_\omega L$ coincide on $L^2(N, \rho)$.  This proves the result.  \end{proof}

It is clear that $P_{\omega}$ is the unique operator $T$ from $\overline{N\xi_{\omega}}$ into $\overline{\xi_{\omega}M}$ that for all $\xi \in \overline{N\xi_{\omega}}$ and $\eta \in \overline{\xi_{\omega}M}$ satisfies \begin{align}\label{Eq: Characterize P}
\langle T \xi,\eta \rangle_{\omega}=\langle \xi,\eta \rangle_{\omega}.
\end{align}

Equation \eqref{Eq: Main Relation} expresses a relationship between $P_\omega$  and $\Phi_\omega$ at the Hilbert space level, but not necessarily at the level of the dense subspaces $N \Omega_\rho$ and $N \xi_\omega.$   Investigating these maps at the algebraic level, one might hope to show that $P_\omega$ is the extension of the map defined for $x\in N$ by $T(x\xi_{\omega})=\xi_{\omega}\Phi_{\omega}(x)$.  It is not clear that this map satisfies \eqref{Eq: Characterize P} for all $\xi \in \overline{N\xi_{\omega}}$ and $\eta \in \overline{\xi_{\omega}M}$, since it is, a priori, unbounded.   However, we will see that modular theory allows us to determine a formula for $P_\omega$ in terms of $\Phi_\omega$ on a dense subalgebra of analytic elements.

Let $\mathfrak{N}$ and $\mathfrak{M}$ be systems as above, and $\Phi \in J_{m}(\mathfrak{N},\mathfrak{M})$ in the following. Consider the set
\vspace{2mm}
\[N_\rho = \{x \in N: \text{ there is a $w^{*}$--entire function } F:\mathbb{C} \rightarrow N \text{ such that } F(it) = \sigma_t^\rho(x)\}.\]

\vspace{2mm}
It is a standard fact, following immediately for example from (16) on page 29 together with (1) on page 32 of \cite{St}, that $N_\rho$ is a $w^*$-dense $\ast$-subalgebra of $N$; it is often called the algebra of \emph{entire analytic} elements of $N$.   If $x \in N_\rho,$ then in fact the function $F_{x}:\Cplx \rightarrow N$ extending the map $it \mapsto \sigma_t^\rho(x)$ is unique, and for $z\in \mathbb{C}$ we define $\sigma_{-iz}^\rho(x)$ to be $F_{x}(z)$.  Denote by $\mathcal{D}_\rho$ the subspace $N_\rho \Omega_\rho \subset L^2(N,\rho).$ Define $M_\varphi \subseteq M$ and $\mathcal{D}_\varphi$ analogously.

The following lemma shows how a $(\rho,\varphi)$--Markov map $\Phi$ connects the respective entire analytic elements of the states $\rho$ and $\varphi,$ and the analytic continuations of the associated maps arising from their respective modular automorphism groups.  We will make use of the resulting symmetry repeatedly in the sequel.

\begin{lemma} \label{ModularSymmetry}Let $\mathfrak{N}=(N,\rho)$ and $\mathfrak{M}=(M,\varphi)$ be von Neumann algebras with faithful normal states, and $\Phi$ a $(\rho,\varphi)$--Markov map.  Denote by $N_\rho \subseteq N$ and $M_\varphi \subseteq M$ the respective subalgebras of entire analytic elements.  Then $\Phi(N_\rho) \subseteq M_\varphi$ and for any $x \in N_\rho$ and $z \in \Cplx,$ we have that 
\[ \Phi \circ \sigma_z^\rho(x) = \sigma_z^\varphi \circ \Phi(x).\]

\end{lemma}

\begin{proof}
It follows from Theorem 3.1 of \cite{BCM2}\footnote{Theorem 3.1 of \cite{BCM2}, although a straightforward observation for unitary $T_{\Phi}$, in the general case requires a long and technical modular theory argument.} that $T_{\Phi_\omega}(\mathcal{D}_\rho) \subseteq \mathcal{D}_\varphi.$ This, together with the fact that $\Omega_{\varphi}$ is separating for $M$, implies that $\Phi_\omega(N_\rho) \subseteq M_\varphi$, and the final equality follows from uniqueness of the entire extension.
\end{proof}
 
 Note that Lemma \ref{ModularSymmetry} implies that for any $x \in N_\rho,$ we have $\sigma_{i/2}^\varphi(\Phi_\omega(x))^{*}=\sigma_{-i/2}^\varphi(\Phi_\omega(x^*))\in M$, and 

\[\Omega_{\varphi}\sigma^{\varphi}_{i/2}(\Phi_{\omega}(x))=J_{\varphi}\sigma_{-i/2}^\varphi(\Phi_\omega(x^*))J_{\varphi} \Omega_\varphi=J_{\varphi}\Delta^{1/2}_\varphi \Phi_\omega(x^*)\Delta^{-1/2}_\varphi J_{\varphi} \Omega_\varphi=\Phi_\omega(x)\Omega_\varphi.\]

From this observation, we obtain the following result of independent interest, which will also be useful in the sequel.

\begin{theorem} \label{Projection Formula} If $\mathfrak{N}=(N,\rho,\alpha,G)$ and $\mathfrak{M}=(M,\varphi,\beta,G)$ are systems, and $\omega \in J_S(\mathfrak{N},\mathfrak{M}),$  then for any $x \in N_\rho$ we have 
\[  P_\omega(x \xi_\omega) = \xi_\omega \sigma_{i/2}^\varphi(\Phi_\omega(x)).\]
\end{theorem}
\begin{proof}Letting both sides of \eqref{Eq: Main Relation} act on $x\Omega_{\rho}\in \mathcal{D}_\rho$ we obtain
\begin{align*}
P_{\omega}(x\xi_{\omega})&=R_{\Omega_{\varphi},\xi_\omega}T_{\Phi_{\omega}}(x\Omega_{\rho})\\
&= R_{\Omega_{\varphi},\xi_\omega} \Phi_{\omega}(x)\Omega_{\varphi}\\
&=  R_{\Omega_{\varphi},\xi_\omega} \Omega_\varphi \sigma_{i/2}^\varphi(\Phi_\omega(x))\\
&=\xi_\omega \sigma_{i/2}^\varphi(\Phi_\omega(x)).
\end{align*}
\end{proof}

\subsection{Joinings over a common subsystem and relative independence}

We now consider the notion of joining of two systems over a common subsystem from the point of view of correspondences. The modular symmetry of the associated Markov maps seen above will be essential for obtaining full generalizations of classical results on relative independence of systems over a common subsystem.
For convenience, we begin by recalling a few key concepts from \cite{BCM}.

\begin{definition} \label{definition:common subsystem} Let $\mathfrak{B}=(B,\mu,\gamma,G)$ and $\mathfrak{N}=(N,\rho,\alpha,G)$ be $W^{*}$-- dynamical systems. Given an injective $\ast$-homomorphism $\iota \in J_{m}(\mathfrak{B},\mathfrak{N})$, the pair $(\mathfrak{B}, \iota)$ is called a subsystem of $\mathfrak{N}$. We call such a map $\iota$ an embedding of the system $\mathfrak{B}$ into $\mathfrak{N}$. If $\mathfrak{M}=(M,\varphi,\beta,G)$ is another $W^{*}$-- dynamical system, and there are embeddings $\iota_N \in J_{m}(\mathfrak{B},\mathfrak{N})$ and $\iota_M \in J_{m}(\mathfrak{B}, \mathfrak{M})$ then the triple $(\mathfrak{B}, \iota_{N}, \iota_{M})$ is called a common subsystem of $\mathfrak{N}$ and $\mathfrak{M}$.
\end{definition}

When the embeddings $\iota_N$ and $\iota_M$ in Definition \ref{definition:common subsystem} are clear from the context, in what follows, we will identify a subsystem $\mathfrak{B}$ with its images in $\mathfrak{N}$ and $\mathfrak{M}$, and simply refer to $\mathfrak{B}$ as a common subsystem of $\mathfrak{M}$ and $\mathfrak{N}$.

Note that if $\mathfrak{N} = (N, \rho, \alpha, G)$ is a system, and $B$ is a von Neumann subalgebra of $N$ which is the image of a normal, $\rho$-preserving conditional expectation $\mathbb{E},$ then also $\sigma^\rho_t \circ \mathbb{E} = \mathbb{E} \circ \sigma^\rho_t$ for all $t \in \Real$.  Since such an expectation is necessarily unique, it also commutes with the automorphisms $\alpha_g$ for all $g \in G.$  Therefore, the action $\alpha$ restricts to $B$ and $\mathfrak{B} = (B, \rho \circ \mathbb{E}, \alpha, G)$ defines a subsystem of $\mathfrak{N}$ via the inclusion map of $B$ into $N.$  On the other hand, if $\mathfrak{B}$ is a subsystem of $\mathfrak{N},$ then the associated embedding $\iota:B \rightarrow N$ necessarily satisfies $\iota \circ \sigma_t^\mu = \sigma_t^\rho \circ \iota$ for all $t \in \Real,$ and it follows that there is a unique normal, $\rho$-preserving conditional expectation $\mathbb{E}: N \rightarrow \iota(B)$ (see, for instance, Theorem 5.4 of \cite{BCM}). The uniqueness of $\mathbb{E}$ implies, furthermore, that $\alpha_g \circ \mathbb{E} = \mathbb{E} \circ \alpha_g$ for all $g \in G,$ so that $G$ acts on the von Neumann subalgebra $\iota(B)$ of $N.$  Finally, by definition $\iota$ satisfies $\mu = \rho \circ \iota$, so we can without loss of any dynamical information identify $B$ with its image $\iota(B)$ in $N$.

The main focus of this paper is on the class of joinings of systems $\mathfrak{N}$ and $\mathfrak{M}$ that ``fix" a common subsystem $\mathfrak{B}$; more specifically, we consider  joinings $\Phi \in J_{m}(\mathfrak{N},\mathfrak{M})$ which intertwine the inclusions $\iota_{N}$ and $ \iota_{M}$ of $B$ into $N$ and $M$. In Proposition \ref{States vs UCP}, we obtain a characterization of such a map $\Phi$ in terms of an associated state on $B \odot B^{op}$.  In the proof, and throughout the remainder of this paper, we will make use of the notion of the adjoint of a ucp Markov map.   In particular, let $(N,\rho)$ and $(M,\varphi)$ be as above, and let $\Phi:N\rightarrow M$ be a normal u.c.p. map. Then, recall  \cite{AC} that $\Phi$ satisfies the conditions
\[ \varphi\circ\Phi=\rho \text{ and }
\Phi\circ\sigma_{t}^{\rho}=\sigma_{t}^{\varphi}\circ\Phi, \quad
t\in\mathbb{R}\]
if and only if there is a normal u.c.p. map $\Phi^{\ast}:M\rightarrow N$ satisfying
\begin{equation}\label{Eq: AccardiCecciniAdjoint}
\rho(\Phi^{\ast}(y)x)=\varphi(y\Phi(x)), \text{ }y\in M\text{ and }x\in N.\\
\end{equation}


\begin{proposition} \label{States vs UCP}  Let $\mathfrak{N} = (N, \rho, \alpha, G)$ and $\mathfrak{M} = (M, \varphi, \beta, G)$ be systems, and let $\mathfrak{B} = (B, \mu, \gamma, G)$ be a common subsystem of $\mathfrak{N}$ and $\mathfrak{M}$.  Denote by $\iota_N$ and $\iota_{M}$  the respective embeddings of $B$ into $N$ and $M$, and $\psi$ the state on $B\odot B^{op}$ defined by
\[ \psi(b_1 \otimes b_2^{op}) = \ip{b_1 \Omega_\mu b_2}{\Omega_\mu}_{\mu}.\]
Let $\Phi \in J_{m}(\mathfrak{N},\mathfrak{M})$ be a joining.  Then, the following conditions are equivalent:

\begin{itemize}
\item [(i)] The restriction of $\Phi$ to $\iota_N(B)$ is the injective $\ast$-homomorphism $\iota_{N}(b)\mapsto \iota_{M}(b)$.

\item [(ii)] The state $\omega = \omega_\Phi$ satisfies $\omega \circ (\iota_N \otimes \iota_M^{op}) = \psi$ on the $\ast$-algebra $B \odot B^{op}$, where $\iota_{M}^{op}$ is the natural map $J_{\mu}b^{*}J_{\mu} \mapsto J_{\varphi} \iota_{M}(b)^{*} J_{\varphi}$ from $B'\cap \mathbf{B}(L^{2}(B,\mu))\rightarrow M'\cap \mathbf{B}(L^{2}(M,\varphi))$.   
\end{itemize}
\end{proposition}

\begin{proof}  First, assuming that condition $(ii)$ holds, observe that the GNS space $\Hil_\omega$ is a $B-B$ bimodule with left-right action given by $b_1 \cdot \zeta \cdot b_2 := \iota_N(b_1) \zeta \iota_M(b_2)$, for all $\zeta \in \mathcal{H}_{\omega}$.  We will show that $L^2(B,\mu)$ is contained in the pointed correspondence $(\Hil_\omega, \xi_\omega)$ as a $B-B$ bimodule.  To do this, define a map $V: L^2(B,\mu) \rightarrow \Hil_\omega$ by first setting
\[ V(b_1 \Omega_\mu b_2) = \iota_N(b_1) \xi_\omega \iota_M(b_2)\]
for $b_1,b_2 \in B$, and then extending to linear combinations in the obvious way.  Using the assumption that $\omega \circ (\iota_N \otimes \iota_M^{op}) = \psi$ on $B \odot B^{op},$ it is then routine to check that $V$ is isometric and satisfies $V(b_1 \zeta b_2) = \iota_N(b_1) V(\zeta) \iota_M(b_2)$ for all $\zeta \in L^2(B,\mu)$ and any $b_1, b_2 \in B.$  This shows that $L^2(B,\mu) \subseteq \Hil_\omega$ as a $B-B$ correspondence.  

Now, if $b \in B_\mu$ is an entire analytic element with respect to the state $\mu$, then $\iota_M(b) \in M_\varphi$ by Lemma \ref{ModularSymmetry}.  Moreover,
\[ \iota_N(b) \xi_\omega = \iota_N(b)V\Omega_\mu = V(\Omega_\mu \sigma_{i/2}^\mu(b)) = \xi_\omega \iota_M( \sigma_{i/2}^\mu(b)) = \xi_\omega \sigma_{i/2}^\varphi(\iota_M(b)).\]
  Denote by $R$ the right exchange map $L^2(M,\varphi) \rightarrow \Hil_\omega$, and $\pi$ the left action of $N$ on $\Hil_\omega$. Recall from \cite{BCM} that $\Phi$ is precisely the normal ucp map associated to the $N-M$ correspondence $\Hil_\omega$, so has the form $\Phi=\Phi_\omega = R^* \pi(\cdot)R.$  To see that $\Phi$ behaves nicely on $\iota_N(B)$, let $b \in B_\mu$, and $y,z \in M$ be arbitrary.  Then
\begin{align*} \ip{\Phi(\iota_N(b))\Omega_\varphi y}{\Omega_\varphi z}_{\varphi} &= \ip{R^* \pi(\iota_N(b)) R(\Omega_\varphi y)}{\Omega_\varphi z}_{\varphi} \\
                                                                &= \ip{\iota_N(b) \xi_\omega y}{\xi_\omega z}_{\omega}\\
                                                                &= \ip{\xi_\omega \sigma^\varphi_{i/2}(\iota_M(b)) y}{\xi_\omega z}_{\omega}\\
                                                                &= \ip{\Omega_\varphi \sigma^\varphi_{i/2}(\iota_M(b)) y}{\Omega_\varphi z}_{\varphi}\\
                                                                &= \ip{\iota_M(b)\Omega_\varphi y}{\Omega_\varphi z}_{\varphi}.
\end{align*}
Therefore, 
\[\Phi(\iota_N(b)) = \iota_M(b) = \iota_M \circ \iota_N^*(\iota_N(b))\]
for all $b \in B_\mu.$ A standard density argument, using normality of these maps, then shows that this equality holds for all $b \in B,$ so $\Phi$ coincides with the $\ast$-homomorphism $\iota_M \circ \iota_N^*$ on the subalgebra $\iota_N(B) \subseteq N$.  This proves $(i).$

Now suppose that condition $(i)$ holds.  Since $\Phi|_{\iota_N(B)}$ is the injective $\ast$-homomorphism $\iota_{M}\circ\iota_{N}^{*}|_{\iota_N(B)}$, we consider the natural $B$-bimodule structure on the Stinespring Hilbert space $\Hil_\Phi$ given by

\[ b_1 \cdot \zeta \cdot b_2 = \iota_N(b_1) \zeta \, (\Phi\circ \iota_N)(b_2)=\iota_N(b_1) \zeta \, \iota_M(b_2),\]
for $b_1,b_2 \in B$ and $\zeta \in \Hil_\Phi.$  Let $\mu$ be the state $\rho|_B: B \rightarrow \Cplx$, and consider the map $W:L^2(B,\mu) \rightarrow \Hil_\Phi$, defined on a dense subspace of $L^2(B,\mu)$ by setting 

\[ W(b_1 \Omega_\mu b_2) = \iota_N(b_1) \xi_\Phi \iota_M(b_2),\]
for $b_1,b_2 \in B$, and extending to linear combinations.  This map is isometric, since for $b_{1}, b_{2}, b_{3}, b_{4}\in B_{\mu}$, we may employ (5) of \cite{BCM}, Lemma \ref{ModularSymmetry} and the special form of $\Phi$ to obtain
\begin{align*}
\langle W(b_{1}\Omega_{\mu}b_{2}), W(b_{3}\Omega_{\mu}b_{4})  \rangle_{\Phi}&= \langle \iota_{N}(b_{1})\xi_{\Phi}\iota_{M}(b_{2}), \iota_{N}(b_{3})\xi_{\Phi}\iota_{M}(b_{4})\rangle_{\Phi}\\
&=\langle \Phi(\iota_{N}(b_{3}^{*}b_{1}))\Omega_{\varphi}\iota_{M}(b_{2}), \Omega_{\varphi}\iota_{M}(b_{4}) \rangle_{\varphi}\\
&=\langle \Omega_{\varphi}\sigma_{i/2}^{\varphi}(\Phi(\iota_{N}(b_{3}^{*}b_{1})))\iota_{M}(b_{2}), \Omega_{\varphi}\iota_{M}(b_{4}) \rangle_{\varphi}\\
&=\langle J_{\varphi}(\sigma_{i/2}^{\varphi}(\Phi(\iota_{N}(b_{3}^{*}b_{1})))\iota_{M}(b_{2}))^{*}J_{\varphi}\Omega_{\varphi},J_{\varphi}(\iota_{M}(b_{4}))^{*}J_{\varphi}\Omega_{\varphi} \rangle_{\varphi}\\
&=\langle \iota_{M}(b_{4})^{*}\Omega_{\varphi}, \iota_{M}(b_{2})^{*}(\sigma_{i/2}^{\varphi}(\Phi(\iota_{N}(b_{3}^{*}b_{1}))))^{*}\Omega_{\varphi} \rangle_{\varphi}\\
&=\langle \sigma_{i/2}^{\varphi}(\Phi(\iota_{N}(b_{3}^{*}b_{1})))\iota_{M}(b_{2}b_{4}^{*})\Omega_{\varphi}, \Omega_{\varphi}\rangle_{\varphi}\\
&= \varphi(\sigma_{i/2}^{\varphi}(\Phi(\iota_{N}(b_{3}^{*}b_{1})))\iota_{M}(b_{2}b_{4}^{*}))\\
&=\varphi(\sigma_{i/2}^{\varphi}(\iota_{M}(b_{3}^{*}b_{1}))\iota_{M}(b_{2}b_{4}^{*}))\\
&=\varphi(\iota_{M}(\sigma_{i/2}^{\mu}(b_{3}^{*}b_{1}))\iota_{M}(b_{2}b_{4}^{*}))\\
&=\varphi\circ\iota_{M}(\sigma_{i/2}^{\mu}(b_{3}^{*}b_{1})b_{2}b_{4}^{*})\\
&=\mu(\sigma_{i/2}^{\mu}(b_{3}^{*}b_{1})b_{2}b_{4}^{*})\\
&=\langle \sigma_{i/2}^{\mu}(b_{3}^{*}b_{1})b_{2}b_{4}^{*}\Omega_{\mu},\Omega_{\mu}\rangle_{\mu}\\
&=\langle b_{4}^{*}\Omega_{\mu},b_{2}^{*}(\sigma_{i/2}^{\mu}(b_{3}^{*}b_{1}))^{*}J_{\mu}\Omega_{\mu}\rangle_{\mu}\\
&=\langle J_{\mu}(\sigma_{i/2}^{\mu}(b_{3}^{*}b_{1})b_{2})^{*}J_{\mu}\Omega_{\mu}, J_{\mu}b_{4}^{*}J_{\mu}\Omega_{\mu} \rangle_{\mu}\\
&=\langle \Omega_{\mu}\sigma_{i/2}^{\mu}(b_{3}^{*}b_{1})b_{2},\Omega_{\mu}b_{4} \rangle_{\mu}\\
&=\langle b_{3}^{*}b_{1}\Omega_{\mu}b_{2},\Omega_{\mu}b_{4} \rangle_{\mu}\\
&=\langle b_{1}\Omega_{\mu}b_{2},b_{3}\Omega_{\mu}b_{4} \rangle_{\mu}.
\end{align*}

The map $W$ hence extends to an isometric embedding of $L^2(B,\mu)$ into $\Hil_\Phi$, satisfying $W(b_1 \eta b_2) = \iota_N(b_1) W(\eta) \iota_M(b_2)$ for all $b_1,b_2 \in B$ and $\eta \in L^2(B,\mu).$ This shows that $L^2(B,\mu) \subseteq \Hil_\Phi$ as $B$--$B$ correspondences.  To see that $(ii)$ holds, note, as was shown in Theorem 4.6 of \cite{BCM}, that the state $\omega = \omega_\Phi: N \odot M^{op} \rightarrow \Cplx$ given by

\[  \omega(x \otimes y^{op}) = \ip{x \xi_\Phi y}{\xi_\Phi}_{\Phi}.\]
satisfies conditions (\ref{Eq: Pushforwards}) and (\ref{Eq: Diagonal Action Preservation}) of Definition \ref{DefinitionJoiningasState}, so $\omega \in J_s(\mathfrak{N},\mathfrak{M}).$   Moreover, for any $b_1, b_2 \in B$, we have 
\begin{align*} \omega(\iota_N(b_1) \otimes \iota_M(b_2)^{op}) = \ip{\iota_N(b_1) \xi_\Phi \iota_M(b_2)}{\xi_\Phi}_{\Phi} &= \ip{W(b_1 \Omega_\mu b_2)}{W\Omega_\mu}_{\Phi}\\
&=\ip{b_1 \Omega_\mu b_2}{\Omega_\mu}_{\mu} = \psi(b_1 \otimes b_2^{op}). \end{align*}
This completes the proof that $(i)$ implies $(ii).$  \end{proof}
We are now ready to state the main definitions of this section. 

\begin{definition} \label{Joining Over Subsystem} Let $\mathfrak{N}=(N,\rho,\alpha,G)$ and $\mathfrak{M}=(M,\varphi,\beta,G)$ be systems, and let $\mathfrak{B}=(B, \mu, \gamma, G)$ be a common subsystem of $\mathfrak{N}$ and $\mathfrak{M}$, with embeddings $\iota_N \in J_m(\mathfrak{B},\mathfrak{N})$ and $\iota_M \in J_m(\mathfrak{B}, \mathfrak{M}).$  We say that $\Phi \in J_m(\mathfrak{N},\mathfrak{M})$  is a joining of $\mathfrak{N}$ and $\mathfrak{M}$ over the common subsystem $\mathfrak{B}$ if the restriction of $\Phi$ to $\iota_N(B)$ is the $\ast$-isomorphism $\iota_{M}\circ \iota_{N}^{-1}$ of $\iota_N(B)$ with $\iota_M(B)$.  Denote by $J_\mathfrak{B}(\mathfrak{N},\mathfrak{M})$ the collection of all joinings of $\mathfrak{N}$ and $\mathfrak{M}$ over the common subsystem $\mathfrak{B}$.
\end{definition}

A few remarks are in order.  First, although the definition of a joining of systems $\mathfrak{N}$ and $\mathfrak{M}$ over a common subsystem $\mathfrak{B}$ depends on the embeddings $\iota_N$ and $\iota_M,$ in this work, there will always be canonical choices of $\iota_N$ and $\iota_M$ and therefore no ambiguity in this definition. In the case of a self-joining $\Phi$ of a system $\mathfrak{N}$ over a single subsystem $(\mathfrak{B}, \iota_{N})$, the above condition simply asserts that $\Phi$ is $B$--bimodular.   

Secondly, note that our definition of subsystem coincides with that of a ``modular subsystem" in \cite{Du}.  Therefore, recalling from \cite{BCM} the one-to-one correspondence between ucp maps $\Phi \in J_m(\mathfrak{N},\mathfrak{M})$ and their associated states $\omega \in J_s(\mathfrak{N},\mathfrak{M}),$ Proposition \ref{States vs UCP} shows that in this setting our defintion of joinings over common subsystems is equivalent to the one found in \cite{Du}.  As we will see in what follows (and, as can also be seen in \cite{Du}), the assumption of modularity of subsystems is crucial to extending any significant classical results about joinings over subsystems to the noncommutative setting.

If $\mathfrak{B}$ is a common subsystem of $\mathfrak{N}$ and $\mathfrak{M}$ as above, then the ucp map $\iota_M \circ \iota_N^*: N \rightarrow M$ satisfies Definition \ref{Joining Over Subsystem}, so in this situation $J_\mathfrak{B}(\mathfrak{N},\mathfrak{M})$ is always nonempty.   The special case in which $\iota_M \circ \iota_N^*$ is the only joining of $\mathfrak{N}$ with $\mathfrak{M}$ over $\mathfrak{B}$ is interesting, so we precise the following definition.




\begin{definition} \label{Relative Independence} Let $\mathfrak{N}=(N,\rho,\alpha,G)$ and $\mathfrak{M}=(M,\varphi,\beta,G)$ be systems, and let $\mathfrak{B}=(B, \mu, \gamma, G)$ be a common subsystem of $\mathfrak{N}$ and $\mathfrak{M}$, with embeddings $\iota_N \in J_m(\mathfrak{B},\mathfrak{N})$ and $\iota_M \in J_m(\mathfrak{B}, \mathfrak{M}).$ The joining $\iota_M \circ \iota_N^* \in J_\mathfrak{B}(\mathfrak{N},\mathfrak{M})$ is called the relatively independent joining of $\mathfrak{N}$ and $\mathfrak{M}$ over $\mathfrak{B}$.  We say that the systems $\mathfrak{N}$ and $\mathfrak{M}$ are relatively independent over $\mathfrak{B}$ (or disjoint over $\mathfrak{B}$) if $J_\mathfrak{B}(\mathfrak{N},\mathfrak{M}) = \sett{\iota_M \circ \iota_N^*}$.
\end{definition}

The reader will surely notice an ambiguity in the above definition: the expression ``$\mathfrak{N}$ and $\mathfrak{M}$ are disjoint over $\mathfrak{B}$" is symmetric in $\mathfrak{N}$ and $\mathfrak{M}$ but the notation $J_\mathfrak{B}(\mathfrak{N},\mathfrak{M}) = \sett{\iota_M \circ \iota_N^*}$ is not obviously so. However, it is true that  $J_\mathfrak{B}(\mathfrak{N},\mathfrak{M}) = \sett{\iota_M \circ \iota_N^*}$ if and only if 
$J_\mathfrak{B}(\mathfrak{M},\mathfrak{N}) = \sett{\iota_N \circ \iota_M^*}$, an observation which is contained in the following basic fact.

\begin{proposition}\label{Adjointissusbsytemjoin}
Let $\mathfrak{B}$ be a common subsystem of two systems $\mathfrak{N}$ and $\mathfrak{M}$. Then $\Phi\in J_{\mathfrak{B}}(\mathfrak{N},\mathfrak{M})$ if and only if  $\Phi^*\in J_{\mathfrak{B}}(\mathfrak{M},\mathfrak{N})$. 
\end{proposition}

\begin{proof}
Let $\omega_\Phi\in J_{s}(\mathfrak{N},\mathfrak{M})$ be the canonical state associated to the ucp map $\Phi$  and let $\psi:B\odot B^{op}\rightarrow\Cplx$ be the state defined by 
\begin{align*}
\psi(b_1\otimes b_2^{op})=\langle b_1\Omega_\mu b_2,\Omega_\mu\rangle_\mu, \text{ }b_1,b_2\in B. 
\end{align*}
Consider the opposite state of $\psi$, defined by $\psi^{op}:B\odot B^{op}\rightarrow \Cplx$ by 
\begin{align*}
\psi^{op}(b_1\otimes b_2^{op}):=\langle b_2\Omega_\mu b_1,\Omega_\mu\rangle_\mu =\psi(b_2\otimes b_1^{op}), \text{ }b_1,b_2\in B.
\end{align*}
Note that $\Phi^{*}\in J_{m}(\mathfrak{M},\mathfrak{N})$, and consider the associated state $\omega_{\Phi^{*}}\in J_{s}(\mathfrak{M},\mathfrak{N})$. By Proposition \ref{States vs UCP}, we have to show that $\omega_{\Phi^*}(\iota_M(b_1)\otimes \iota_N^{op}(b_2^{op}))=\psi^{op}(b_1\otimes b_2^{op})$ for all $b_1,b_2\in B$. 

Denote by $B_\mu \subseteq B, N_\rho \subseteq N, M_\varphi \subseteq M$ the $*$-algebras of analytic elements for the respective modular automorphism groups $(\sigma_t^\mu),(\sigma_t^\rho)$ and $(\sigma_t^\varphi)$. It follows from \cite[Lemma 2.2]{BCM2} that if  $b \in B_\mu$, then $\iota_N(b)\in N_\rho$ and $\iota_M(b) \in M_\varphi$. Let $b_1,b_2\in B_\mu$. We then have
\begin{align*}
\omega_{\Phi^*}\Big(\iota_M(b_1)\otimes \iota_N^{op}(b_2^{op})\Big)&=\langle\Phi^*(\iota_M(b_1))\Omega_\rho \iota_N(b_2),\Omega_\rho\rangle_\rho \\
&= \langle\Phi^*(\iota_M(b_1))J_\rho (\iota_N(b_2))^*J_\rho \Omega_\rho ,\Omega_\rho\rangle_\rho\\
&= \langle\Phi^*(\iota_M(b_1))\sigma_{-i/2}^\rho(\iota_N(b_2)) \Omega_\rho ,\Omega_\rho\rangle_\rho\\
&=\rho\Big(\Phi^*(\iota_M(b_1))\sigma_{-i/2}^\rho(\iota_N(b_2))\Big)\\
&=\varphi\Big(\iota_M(b_1)\Phi\big(\sigma_{-i/2}^\rho(\iota_N(b_2))\big)\Big)\text{ \indent(by Eq. \eqref{Eq: AccardiCecciniAdjoint})}\\
&=\varphi\Big(\iota_M(b_1)\sigma_{-i/2}^\varphi\big(\Phi(\iota_N(b_2)\big)\Big) \text{ \indent \cite[Lemma 2.2]{BCM2}} \\
&=\langle\Delta_\varphi^{\frac{1}{2}}\Phi(\iota_N(b_2))\Omega_\varphi, (\iota_M(b_1))^*\Omega_\varphi\big)\rangle_\varphi\\
&=\langle\Phi(\iota_N(b_2))\Omega_\varphi, \Delta_\varphi^{\frac{1}{2}}(\iota_M(b_1))^*\Omega_\varphi\big)\rangle_\varphi\\
&=\langle\Phi(\iota_N(b_2))\Omega_\varphi, \sigma_{-i/2}^\varphi((\iota_M(b_1))^*)\Omega_\varphi\big)\rangle_\varphi\\
&=\langle\Phi(\iota_N(b_2))\Omega_\varphi, J_\varphi(\iota_M(b_1))J_\varphi\Omega_\varphi\rangle_\varphi \text{ \indent(\cite{Fa})}\\
&=\langle\Phi(\iota_N(b_2))J_\varphi(\iota_M(b_1))^*J_\varphi\Omega_\varphi, \Omega_\varphi\rangle_\varphi\\
&=\langle\Phi(\iota_N(b_2))\Omega_\varphi\iota_M(b_1), \Omega_\varphi\rangle_\varphi\\
&=\psi(b_2\otimes b_1^{op}) \text{ \indent(by Proposition \ref{States vs UCP})}\\
&=\psi^{op}(b_1\otimes b_2^{op}).
\end{align*}
Since $\omega_{\Phi^{*}},\psi,\psi^{op}$ are normal in each variable, the requirement of Proposition \ref{States vs UCP} is established  by invoking density of analytic elements.  \end{proof}

\section{Relative independence and orthogonality}
\label{section:relative independence and orthogonality}
In this section we establish some basic technical results on joinings relative to a common subsystem that will be useful in the sequel.  Among these is a characterization in Theorem \ref{Partial Isometry} of the relatively independent joining in terms of a `projection property' of equivariant Markov maps.  In the proof, we will make use of the notion of the multiplicative domain of a unital completely positive map; we recall the relevant details on multiplicative domains here.  

\begin{definition} Let $A$ and $B$ be unital C$^*$-algebras and $\Phi: A \rightarrow B$  a unital, completely positive map.  The multiplicative domain of $\Phi$ is the set
\[ A_\Phi = \sett{a \in A: \Phi(a^*a) = \Phi(a)^*\Phi(a),\,\Phi(aa^*)=\Phi(a)\Phi(a)^*} \subseteq A.\]
\end{definition}

\smallskip

The following result of Choi gives a further characterization of the multiplicative domain.  

\begin{lemma}\label{Bimodules Homomorphisms Choi} \cite{Ch} If $A$ and $B$ are unital C$^*$-algebras and $\Phi:A \rightarrow B$ is a ucp map, then $A_\Phi$ is a C$^*$-subalgebra of $A$, the restriction of $\Phi$ to $A_\Phi$ is a $\ast$-homomorphism, and furthermore,
\[A_\Phi = \sett{a \in A: \Phi(ax)=\Phi(a)\Phi(x),\,\Phi(xa) = \Phi(x)\Phi(a) \text{ for all } x \in A}.\]
\end{lemma}

Observe that, using the defining property of the adjoint, it is straightforward to check that the map $\iota_N \circ \iota_N^*: N \rightarrow \iota_N(B)$ is an $\iota_N(B)$-bimodule map, and preserves the state $\rho$.  By uniqueness of the conditional expectation, it follows that $\iota_N \circ \iota_N^* = \mathbb{E}_{\iota_N(B)}$, i.e. $\iota_N^* = \iota_N^{-1}\circ \mathbb{E}_{\iota_N(B)}$.  The following result characterizes precisely when a joining between systems $\mathfrak{N}$ and $\mathfrak{M}$ gives rise to a common subsystem of $\mathfrak{N}$ and $\mathfrak{M}$ and yields a new proof of the ``orthogonality" characterization of the relatively independent joining mentioned above.

\begin{theorem} \label{Partial Isometry} Let $\mathfrak{N} = (N, \rho, \alpha,G)$ and $\mathfrak{M} = (M, \varphi, \beta, G)$ be systems, and let $\Phi \in J_m(\mathfrak{N},\mathfrak{M})$. If $(\mathfrak{B}, \iota_N, \iota_M)$ is a common subsystem of $\mathfrak{N}$ and $\mathfrak{N}$ then the relatively independent joining $\Phi = \iota_M \circ \iota_M^*$ satisfies $\Phi = \Phi \Phi^* \Phi$.  On the other hand, if $\mathfrak{N}$ and $\mathfrak{M}$ are systems and $\Phi \in J_m(\mathfrak{N}, \mathfrak{M})$ satisfies $\Phi = \Phi \Phi^* \Phi$, then $\mathfrak{N}$ and $\mathfrak{M}$ have a common subsystem $\mathfrak{B}$, and $\Phi$ is the relatively independent joining of $\mathfrak{N}$ and $\mathfrak{M}$ over $\mathfrak{B}$.  

\end{theorem}

\begin{proof} Suppose first that $\Phi \in J_m(\mathfrak{N},\mathfrak{M})$ satisfies $\Phi = \Phi \Phi^* \Phi$.  By Lemma 6.3 of \cite{BCM}, the algebra of harmonic elements (cf. \cite{Iz})
\begin{align}\label{FixB}
B&=\{x\in N:\Phi^{*}\Phi(x)=x\}\\&=\{x\in N: \Phi^{*}\Phi(xy)=x\Phi^{*}\Phi(y), \text{ } \Phi^{*}\Phi(yx)=\Phi^{*}\Phi(y)x \text{ }\forall y\in N\} \nonumber 
\end{align}
is a von Neumann subalgebra of $N$.  Moreover, $\Phi = \Phi\Phi^*\Phi$ implies that $(\Phi^*\Phi)^{2}=\Phi^*\Phi$ which, in turn, implies $\Phi^*\Phi(N)=B$. So $\Phi^*\Phi:N\rightarrow N$ is a projection of norm one whose image is $B$ and is also a $B$--bimodule map. By a well--known theorem of Tomiyama \cite{To}, $\Phi^*\Phi$ is a conditional expectation onto $B$. But since $\rho\circ \Phi^*\Phi =\rho$, by \cite{Ta} it follows that $\Phi^*\Phi=\mathbb{E}_{B}$, where $\mathbb{E}_{B}$ is the unique normal $\rho $--preserving conditional expectation onto $B$. Uniqueness of $\mathbb{E}_B$ implies that $B$ is $\alpha$-invariant, so this yields a subsystem $\mathfrak{B} = (B, \rho|_B, \alpha|_B, G)$ of $\mathfrak{N}$.  

By our hypothesis, and  properties of the adjoint (see \cite{BCM} for details), the joining $\Phi^* \in J_m(\mathfrak{M}, \mathfrak{N})$ satisfies $\Phi^* = \Phi^* \Phi \Phi^*,$ so the same argument from the previous paragraph shows that $\Phi \Phi^*: M \rightarrow M$ is the unique normal, $\varphi$-preserving conditional expectation onto the $\beta$-invariant von Neumann subalgebra $Q = \Phi \Phi^*(M)$ of $M$. By the Kadison-Schwarz inequality, and the properties of $\Phi,$ for any $x \in B$ we have 
\[ \Phi(x)^*\Phi(x) \leq \Phi(x^*x) = \Phi(\Phi^*\Phi(x)^*\Phi^*\Phi(x)) \leq \Phi(\Phi^*(\Phi(x)^*\Phi(x))) \leq \Phi(x)^*\Phi(x),\]
so that $\Phi(x^*x) = \Phi(x)^* \Phi(x),$ and by Lemma \ref{Bimodules Homomorphisms Choi} the restriction of $\Phi$ to $B$ is a $\ast$-homomorphism.  Moreover, $\Phi(B) = Q$, since 
\[Q = \Phi \Phi^*(M) \supset \Phi \Phi^* \Phi(B) = \Phi(B) = \Phi \Phi^* \Phi(N) = \Phi(N) \supset \Phi \Phi^*(M) = Q.\]
Note that $\Phi|_B$ is isometric, since $\norm{x} = \norm{\Phi^*\Phi(x)} \leq \norm{\Phi(x)} \leq \norm{x}$ for any $x \in B.$  Thus, the restriction of $\Phi$ to $B$ defines an embedding $\iota_M$ of $B$ into $M$ which, by construction, satisfies the necessary equivariance relations for membership in $J_m(\mathfrak{B},\mathfrak{M})$.  Thus, $\mathfrak{B}$ is a common subsystem of $\mathfrak{N}$ and $\mathfrak{M}.$  Finally, if we let $\iota_N$ be the inclusion of $B$ into $N$,  since $\Phi^*\Phi = \mathbb{E}_{\iota_N(B)}$, $ \iota_{M}=\Phi\circ\iota_{N}$ and $\iota_N^* = \iota_N^{-1} \circ \mathbb{E}_{\iota_N(B)}$, we see that $$\Phi = \Phi \Phi^*\Phi =\Phi\circ\mathbb{E}_{\iota_{N}(B)}=\Phi\circ\iota_{N}\circ \iota_{N}^{-1}\circ\mathbb{E}_{\iota_{N}(B)}= \iota_M \circ \iota_N^*,$$ the relatively independent joining of $N$ with $M$ over $B$.    
   
   Conversely, suppose that $\mathfrak{B} = (B, \mu, \gamma, G)$ is a common subsystem of $\mathfrak{N}$ and $\mathfrak{M}$.  Write $\iota_N$ and $\iota_M$ for the associated embeddings, and let $\Phi = \iota_M \circ \iota_N^*$ denote the relatively independent joining.  Then
   \begin{align*} \Phi^*\Phi = \iota_N\circ\iota_M^*\circ\iota_M\circ\iota_N^* &=
  \iota_N \circ \iota_M^{-1} \circ \E_{\iota_M(B)} \circ \iota_M \circ \iota_N^{-1} \circ \E_{\iota_N(B)}\\
  &=\iota_N \circ \iota_N^{-1} \circ \E_{\iota_N(B)}\\
  &=\E_{\iota_N(B)}, \end{align*}
  from which it follows that 
  \[\Phi\Phi^*\Phi = \iota_M \circ \iota_N^{-1} \circ \E_{\iota_N(B)} = \Phi.\]
  This completes the proof.  \end{proof}

As a corollary, we use our description of the projection $P_\Phi$ associated to a joining $\Phi$ obtained in the previous section to give the following Hilbert space formulation of disjointness relative to a subsystem.

\begin{theorem} \label{relative independence characterization} Let $\mathfrak{N}$ and $\mathfrak{M}$ be systems, and $\mathfrak{B}$ a common substem of $\mathfrak{M}$ and $\mathfrak{N}$, as above.  Let $\Phi \in J_\mathfrak{B}(\mathfrak{N},\mathfrak{M}),$ and let $P_\Phi: \overline{N \xi_\Phi} \rightarrow  \overline{\xi_\Phi M}$ be the restriction of the orthogonal projection of $\Hil_\Phi$ onto $\overline{\xi_{\Phi}M}.$  Then the following conditions are equivalent:
\begin{itemize}
\item [(i)] The map $\Phi$ is the relatively independent joining of $\mathfrak{N}$ and $\mathfrak{M}$ over $\mathfrak{B}$.
\item[(ii)] The subspaces $\overline{N \xi_\Phi} \ominus \overline{\iota_N(B)\xi_\Phi}$ and $\overline{\xi_\Phi M} \ominus \overline{\xi_\Phi \iota_M(B)}$ are orthogonal in $\Hil_\Phi$. 
\item[(iii)] $P_\Phi(\overline{N \xi_\Phi}) \subseteq \overline{\xi_\Phi \iota_M(B)}$
\end{itemize}
\end{theorem}
\begin{proof} First, note that if $W:\mathcal{H}_{\omega_{\Phi}}\rightarrow \mathcal{H}_{\Phi}$ is the natural pointed bimodule map extending the identification of canonical cyclic vectors, it is an isomorphism by (2) of Theorem 4.6 of \cite{BCM}. Since $W^{*}|_{\overline{N\xi_{\Phi}}}:\overline{N\xi_{\Phi}}\rightarrow \overline{N\xi_{\omega_{\Phi}}}$ is a left Hilbert $N$--module isomorphism and $W|_{\overline{\xi_{\omega_{\Phi}}M}}:\overline{\xi_{\omega_{\Phi}}M}\rightarrow \overline{\xi_{\Phi}M}$ is a right Hilbert $M$--module isomorphism, $WP_{\omega_{\Phi}}W^{*}|_{\overline{N\xi_{\Phi}}}$ is the projection of $\overline{N\xi_{\Phi}}$ onto $\overline{\xi_{\Phi}M}$ and hence $W$  intertwines $P_{\omega_{\Phi}}$ and $P_{\Phi}$, and the immediate analogue of Theorem \ref{Projection Formula} holds for $P_{\Phi}$. 

The proof that $(ii)$ implies $(iii)$ is a consequence of Theorem \ref{Projection Formula} applied to dense sets of analytic vectors together with standard Hilbert space theory: Assuming $(ii)$, given $\eta \in \overline{\xi_\Phi M} \ominus \overline{\xi_\Phi \iota_M(B)}$, one has that $$\langle P_{\Phi}\xi_{0},\eta \rangle=\langle \xi_{0},\eta \rangle=0$$ for any $\xi_{0} \in \overline{N \xi_\Phi} \ominus \overline{\iota_N(B)\xi_\Phi}$. If $\xi\in \overline{N \xi_\Phi}$ decomposes as $\xi_{1}\oplus \xi_{0}$ with $\xi_{1}\in \overline{\iota_{N}(B)\xi_{\Phi}}$ and $\xi_{0} \in \overline{N \xi_\Phi} \ominus \overline{\iota_N(B)\xi_\Phi}$ then for an $\eta$ as above, $$\langle P_{\Phi}\xi,\eta \rangle=\langle P_{\Phi}\xi_{1},\eta \rangle=0$$ since Theorem \ref{Projection Formula} applied to the algebra $\iota_{N}(B_{\mu})$ of analytic elements implies $P_{\Phi}(\overline{\iota_{N}(B)\xi_{\Phi}})=\overline{\xi_{\Phi}\iota_{M}(B)}$. Indeed, $B_{\mu}\Omega_{\mu}$ is a dense subspace of $L^{2}(B,\mu)=\overline {B\Omega_{\mu}}$ and the map $b\Omega_{\mu}\mapsto \iota_{N}(b)\Omega_{\varphi}$ extends to a left Hilbert module isomorphism of $\overline {B\Omega_{\mu}}$ onto $\overline{\iota_{N}(B)\Omega_{\rho}}$ which, in turn, is isomorphic as a left Hilbert module to $\overline{\iota_{N}(B)\xi_{\Phi}}$ via the right exchange map. 

We now prove that $(i)$ implies $(ii)$. If we denote by $\mathbb{E}$ the unique $\rho$--preserving conditional expectation and let
\[ N_\rho \ominus \iota_N(B) = \sett{x \in N_\rho: \mathbb{E}_{\iota_N(B)}(x) = 0},\]
then $(N_\rho \ominus \iota_N(B))\xi_\Phi$ is a subspace of $\overline{N \xi_\Phi} \ominus \overline{\iota_N(B)\xi_\Phi}$, which is dense in the norm on $\Hil_\Phi$ since every $x\in N_{\rho}$ may be written as $(x-\mathbb{E}_{\iota_{N}(B)}(x))+\mathbb{E}_{\iota_{N}(B)}(x)$ and if $N_{\rho}\xi_{\Phi}\ni x_{\lambda}\xi_{\Phi} \rightarrow \eta\in \overline{N\xi_{\Phi}}$, we have that $(x_{\lambda}-\mathbb{E}_{\iota_{N}(B)}(x_\lambda)))\xi_{\Phi}\rightarrow P\eta$ and $\mathbb{E}_{\iota_{N}(B)}(x_\lambda)\rightarrow (1-P)\eta$, where $P$ is the orthogonal projection of $\mathcal{H}_{\Phi}$ onto $\overline{\iota_{N}(B)\xi_{\Phi}}$.  Similarly, the subspace $\xi_\Phi(M_\varphi \ominus \iota_M(B))$ is norm-dense in $\overline{\xi_\Phi M} \ominus \overline{\xi_\Phi \iota_M(B)}$.  Moreover, if $n \in N_\rho \ominus \iota_N(B)$ and $m \in M_\varphi \ominus \iota_M(B)$ are arbitrary, then by Theorem \ref{Projection Formula} we have
\begin{align*} \ip{n \xi_\Phi}{\xi_\Phi m}_{\omega} = \ip{P_{\Phi}(n \xi_\Phi)}{\xi_\Phi m}_{\omega} &=\ip{\xi_\Phi \sigma_{i/2}^\varphi(\Phi(n))}{\xi_\Phi m}_{\omega} \\
&=\ip{\xi_\Phi \Phi(\sigma^\rho_{i/2}(n))}{\xi_\Phi m}_{\omega} = 0,
\end{align*}
since $\Phi(N) \subseteq \iota_M(B).$  It follows that $\overline{N \xi_\Phi} \ominus \overline{\iota_N(B)\xi_\Phi}$ and $\overline{\xi_\Phi M} \ominus \overline{\xi_\Phi \iota_M(B)}$ are orthogonal.  

We now prove that $(iii)$ implies $(i).$ First note that $(iii)$ and Theorem \ref{Projection Formula} imply that for any $n \in N_\rho$ we have
\[ P_{\Phi}(n\xi_\Phi) = \xi_\Phi \sigma_{i/2}^\varphi(\Phi(n)) = \xi_\Phi (\Phi \circ \sigma_{i/2}^\rho)(n) \in \xi_\Phi \iota_M(B),\]
and it follows, since the associated right exchange map is a unitary right module isomorphism, that $\Phi(N) \subseteq \iota_M(B).$ Since $\Phi|_{\iota_{N}(B)}$ is a $*$--isomorphism of $\iota_{N}(B)$ onto $\iota_{M}(B)$, we have that $$T_{\Phi}|_{\overline{\iota_{N}(B)\Omega_{\rho}}}:\overline{\iota_{N}(B)\Omega_{\rho}}\rightarrow \overline{\iota_{M}(B)\Omega_{\varphi}}$$ is a unitary operator and additionally, by Eq. 8 of \cite{BCM}, $$(T_{\Phi}|_{\overline{\iota_{N}(B)\Omega_{\rho}}})^{*}=T_{\Phi}^{*}|_{\overline{\iota_{M}(B)\Omega_{\varphi}}}=T_{\Phi^{*}}|_{\overline{\iota_{M}(B)\Omega_{\varphi}}},$$ so for all $b\in B$,
\[\iota_{M}(b)\Omega_{\varphi}=(T_{\Phi}|_{\overline{\iota_{N}(B)\Omega_{\rho}}})(T_{\Phi}|_{\overline{\iota_{N}(B)\Omega_{\rho}}})^{*}\iota_{M}(b)\Omega_{\varphi}=\Phi\Phi^{*}(\iota_{M}(b))\Omega_{\varphi}\]
This, together with the fact that $\Omega_{\varphi}$ is separating for $M$ and $\Phi(N)\subseteq \iota_{M}(B)$, yields  
\[\Phi \Phi^* \Phi(n) = \Phi(n)\]
for all $n\in N$. By Theorem \ref{Partial Isometry} we get that $\Phi$ is the relatively independent joining. \end{proof}


\section{Relative ergodicity and mixing properties}
\label{section:relative ergodicity and mixing properties}

In this section, we consider relativized versions of the
notions of ergodicity and primeness of W$^*$-dynamical systems studied in \cite{BCM}.  We also employ the technical observations of the previous section in analyzing the relationship between the properties of compactness and weak mixing of a dynamical system. Our main results in the latter part of this section extend those in \cite{Du2} from abelian groups to general groups, and thereby establish noncommutative generalizations of classical characterizations of weak mixing and compactness in terms of joinings over subsystems.

\subsection{Relative ergodicity and primeness}    The concept of a dynamical system which is ergodic relative to a subsystem has been widely studied and appears throughout the literature in both the classical and noncommutative settings.  Prime dynamical systems, i.e., those that admit no nontrivial subsystem, have also appeared widely in the classical setting. For example, note the interesting fact that an ergodic measure preserving $\mathbb{Z}$--system admits a topological, minimal, prime model if and only if it has zero entropy \cite{We}.  Although we were not able to find a specific reference in the literature to the term `relative primeness' defined below, we find it to be a convenient term for this natural generalization of the notion of a prime dynamical system.

\begin{definition}\label{relative_ergodicity}
Let $\mathfrak{B}$ be a subsystem of $\mathfrak{N}$ with respect to the embedding $\iota_N^B$. 

\noindent $(i)$ We say that the system $\mathfrak{N}$  is an identity system relative to $\mathfrak{B}$ if, for any $x \in N$ with $\E_{\iota_N^B(B)}(x)=0$, one has $\alpha_g(x)=x$ for all $g\in G$.

\noindent $(ii)$ The system $\mathfrak{N}$ is said to be ergodic relative $\mathfrak{B}$ if $N^G\subseteq \iota_N^B(B)$. 

\noindent $(iii)$ We say that $\mathfrak{N}$ is prime relative to $\mathfrak{B}$ if, whenever $\mathfrak{A}=(A,\phi,\delta,G)$ is a subsystem of $\mathfrak{N}$ with embedding $\iota_N^A$ such that $\mathfrak{B}$ is a subsystem of $\mathfrak{A}$ with embedding $\iota_A^B$ and 
\[
  \begin{tikzcd}
    B \arrow{r}{\iota^B_A} \arrow[swap]{dr}{\iota^B_N}
 & A \arrow{d}{\iota^A_N} \\
     & N
  \end{tikzcd}
\]
commutes, then $\iota_N^B(B)=\iota_N^A(A)$.
\end{definition}

The following characterization of relative ergodicity of $W^*$-dynamical systems in terms of relative disjointness was recently proved by Duvenhage in \cite{Du2}; we give a new proof of the result using the machinery of correspondences, as a straightforward corollary to Theorem \ref{relative independence characterization}.  

\begin{theorem}
\label{characterize_relative_ergodic}
Let $\mathfrak{B}$ be a subsystem of $\mathfrak{N}$ with respect to the embedding $\iota_N$.  Then $\mathfrak{N}$ is ergodic relative to $\mathfrak{B}$ if and only if for any system $\mathfrak{M}$ which is an identity system relative to $\mathfrak{B}$, with associated inclusion $\iota_M$, one has $J_{\mathfrak{B},m}(\mathfrak{N},\mathfrak{M})=\sett{\iota_{M}\circ\iota_{N}^*}$, i.e., $\mathfrak{N}$ and $\mathfrak{M}$ are disjoint relative to $\mathfrak{B}$. 
\end{theorem}

\begin{proof}  First let $\mathfrak{N} = (N, \rho, \alpha, G)$ be a system which is ergodic relative to a subsystem $\mathfrak{B}=(B,\iota_N)$, and suppose that $\mathfrak{M} = (M, \varphi, \beta, G)$ is a system which also contains $\mathfrak{B}$ as a subsystem, and is the identity relative to $\mathfrak{B}$. Let $\Phi: M \rightarrow N$ be an element of $J_{\mathfrak{B},m}(M,N)$.  Then, for any $y \in M \ominus \iota_M(B)$, since $\beta_g(y)=y$ for all $g \in G$ we must also have
\[\alpha_g(\Phi(y)) = \Phi(\beta_g(y)) = \Phi(y), \quad g \in G,\]
and then ergodicity of $\mathfrak{N}$ relative to $\mathfrak{B}$ implies $\Phi(y) \in \iota_N(B)$.  By definition, we also have $\Phi(\iota_M(B)) \subseteq \iota_N(B)$, so this means the image of $\Phi$ is contained in $\iota_N(B)$.  Then, for any analytic element $m \in M$, by Theorem \ref{Projection Formula} one has
\[P_\Phi(m \xi_\Phi) = \xi_\Phi \sigma_{i/2}^\varphi (\Phi(m)) \in \xi_\Phi \iota_N(B),\]
owing to invariance of $\iota_N(B)$ under $\sigma_{i/2}^\varphi$.  A straightforward density argument then shows that  Theorem \ref{relative independence characterization} applies, and we see that $J_{\mathfrak{B},m}(\mathfrak{M},\mathfrak{N}) = \sett{\iota_N \circ \iota_{M}^*}$, i.e., $\mathfrak{M}$ and $\mathfrak{N}$ are disjoint relative to $\mathfrak{B}$.

Conversely, suppose that $\mathfrak{N} = (N, \rho, \alpha, G)$ is not ergodic relative to a subsystem $\mathfrak{B} = (B, \iota_N)$.  Then there is some $z \in N$ for which $\alpha_g(z)=z$ for all $g \in G$ while $\E_{\iota_N(B)}(z) = 0$.  Then the von Neumann subalgebra $P$ of $N$ generated by $\iota_N(B)$ and $z$ strictly contains $\iota_N(B)$, and defines a subsystem $\mathfrak{P}$ of $\mathfrak{N}$ via the inclusion $ \iota: P \hookrightarrow N$. Moreover, it is easy to see that $\mathfrak{P}$ is an identity system relative to $\mathfrak{B}$.  But now the conditional expectation of $N$ onto $P$ defines a joining of $\mathfrak{N}$ and $\mathfrak{P}$ over the common subsystem $\mathfrak{B}$ which is not the relatively independent joining.
\end{proof}
We also obtain the following characterization of relative primeness, in terms of an ergodicity property of joinings.

\begin{theorem} Let $\mathfrak{B}$ be a subsystem of  $\mathfrak{N}$.  Then $\mathfrak{N}$ is prime relative to $\mathfrak{B}$ if and only if for any non-identity map $\Phi \in J_{\mathfrak{B}}(\mathfrak{N},\mathfrak{N}),$  every element $x \in N$ which satisfies $\Phi(x)=x$ is contained in $\iota_N(B)$. 
\end{theorem}
\begin{proof} First assume that $\mathfrak{N}=(N, \rho, \alpha, G)$ is prime relative to the subsystem $\mathfrak{B}=(B, \iota_N^B)$, and suppose $\Phi \in J_\mathfrak{B}(\mathfrak{N},\mathfrak{N})$ is not the identity.  The set of ``harmonic elements" $N^\Phi = \sett{x \in N: \Phi(x)=x}$ is a von Neumann subalgebra of $N$ (see \cite[Theorem 6.3]
{BCM}) containing  $\iota_N^B(B)$, which is invariant under both $\alpha$ and the modular automorphism group $(\sigma_t^\rho)_{t \in \Real}$.  Therefore, the inclusion of $N^\Phi$ into $N$ makes $\mathfrak{A} = (N^\Phi, \rho|_{N^\Phi}, \alpha|_{N^\Phi}, G)$  a subsystem of $\mathfrak{N}$.  Primeness of $\mathfrak{N}$ relative to $\mathfrak{B}$ then implies $N^\Phi \subseteq \iota_N^B(B)$.  

Conversely, suppose that for any nonidentity map $\Phi \in J_\mathfrak{B}(\mathfrak{N}, \mathfrak{N})$, we have $N^\Phi \subseteq \iota_N^B(N)$.  Let $\mathfrak{A} = (A, \iota_N^A)$ be a subsystem of $\mathfrak{N}$ which also contains $\mathfrak{B}$ as a subsystem, in such a way that $\iota_N^B = \iota_N^A \circ \iota_A^B.$  Then there is a conditional expectation $\Phi$ of $N$ onto $\iota_N^A(A)$ which intertwines both $\alpha$ and the modular automorphisms $\sigma_t^\rho, t \in \Real$.  This map $\Phi$ defines an element of $J_\mathfrak{B}(\mathfrak{N}, \mathfrak{N})$ which is not the identity map.  But then our hypothesis implies $\iota_N^A(A) \subseteq N^\Phi \subseteq \iota_N^B(B)$, from which it follows that $\iota_N^A(A) = \iota_N^B(B).$  Thus, $\mathfrak{N}$ is prime relative to the subsystem $\mathfrak{B}$. \end{proof}

\subsection{Compactness and weak mixing}
We begin by recalling the definitions of the two main properties under consideration in what follows.

\begin{definition}
Let $\mathfrak{M} = (M,G, \beta, \varphi)$ be a system. An element $y\in M$ is said to be weakly mixing for $\beta$ if for every finite subset $F$ of $M\ominus \mathbb{C}1$ and every $\epsilon>0$ there exists a $g\in G$ such that
$|\varphi(x^{*}\beta_{g}(y))|<\epsilon$ for all $x\in F$.  The system $\mathfrak{M}$ is called weakly mixing if every $y \in M \ominus \mathbb{C}1$ is weakly mixing for $\beta$.  
\end{definition}

\begin{definition} A system $\mathfrak{M} = (M, \varphi, \alpha, G)$ is said to be compact if, for every $x \in M$, the orbit $\sett{\alpha_g(x)\Omega_\varphi: g \in G}$ is precompact in $L^2(M, \varphi)$. 
\end{definition}

One of the main results obtained \cite{BCM} was that an ergodic W$^*$-dynamical system is weakly mixing precisely when it admits no nontrivial compact subsystem.  More specifically, Lemma 6.12 of \cite{BCM} shows that if $\mathfrak{M} = (M,G,\beta,\varphi)$ is a W$^*$-dynamical system  which is ergodic, but not weakly mixing, then any finite-dimensional subspace $F$ of $M$ which is invariant under $\beta$ yields a compact subsystem of $\mathfrak{M}$ by letting $G$ act on the von Neumann algebra generated by $F$ and the identity $1 \in M$.   It is then natural from a dynamical point of view to consider whether there is a ``maximal" object of this form, and what its properties are.  To this end, we consider the following definition, which is analogous to the classical definition\footnote{In the classical setting the Kronecker system is the sub $\sigma$--algebra spanned by the eigenfunctions of the action. That definition is inappropriate in general, since such eigenfunctions do not always exist.}

\begin{definition}  Let $\mathfrak{M}=(M,G,\beta,\varphi)$ be an ergodic W$^*$-dynamical system.  Define the Kronecker subalgebra to be the von Neumann subalgebra $M_{K,\beta}$ of $M$ generated by the union of all finite-dimensional $\beta$-invariant subspaces of $M$.  When the action is clear from context (which it always will be), we will abbreviate this notation to $M_K$.  
\end{definition}

We first use the main results of \cite{BCM} to show that the Kronecker subalgebra associated to a system $\mathfrak{M}$ defines a compact subsystem of $\mathfrak{M}$.

\begin{theorem} \label{theorem:Kronecker} If $\mathfrak{M}=(M,G,\beta,\varphi)$ is an ergodic system, then the Kronecker subalgebra $M_K$ is injective and tracial, and defines the unique maximal compact subsystem of $\mathfrak{M}$.
\end{theorem}

\begin{proof}
The unitary representation $V$ induced by the action  $\beta$ on $M$ splits into a direct sum of compact and weakly mixing parts, i.e. $L^{2}(M,\varphi)$ splits into a direct sum of $V$--invariant subspaces $\mathcal{H}_{c}\oplus \mathcal{H}_{wm}$ such that $\mathcal{H}_{c}$ is the maximal closed subspace in which the $V$--orbit of any vector is $\|\cdot\|_{\varphi}$--precompact. The strong-operator closure of the image of $V|_{\mathcal{H}_{c}}$ is then (for instance, by Theorem 3.7 of \cite{Gl}) a compact topological group. 

Let $M_{0}:=vN(\{x\in M:x\Omega_{\varphi}\in\mathcal{H}_{c}\cap M\Omega_{\varphi}\})$. Since it is generated by a $\beta$--invariant subset, the von Neumann algebra $M_{0}$ is $\beta$--invariant. We will employ Theorem 6.9 of \cite{BCM} to show that $M_{0}\subset M^{\varphi}$ and that any element of $M_{0}$ has a $\norm{\cdot}_{\varphi}$--precompact orbit. To begin with, note that by the Peter-Weyl Theorem and Theorem 6.9 of \cite{BCM}, $\mathcal{H}_{c}$ splits as an orthogonal direct sum $\bigoplus_{i=1}^{\infty}S_{i}\Omega_{\varphi}$ with $S_{i}\subseteq M^{\varphi}$ finite-dimensional $\beta$--invariant subspaces of $M^{\varphi}$. Thus $\mathcal{H}_{c}\subseteq L^{2}(M^{\varphi},\varphi)$ and therefore $\{x\in M:x\Omega_{\varphi}\in\mathcal{H}_{c}\cap M\Omega_{\varphi}\}\subseteq M^{\varphi}$ and consequently $M_{0}\subset M^{\varphi}$. It is straightforward to show (cf. the proof of Lemma 6.12 of \cite{BCM}) that any element $y$ in the $*$-algebra $M_{00}$ generated by $\{x\in M:x\Omega_{\varphi}\in\mathcal{H}_{c}\cap M\Omega_{\varphi}\}$ and $1$ will have $\norm{\cdot}_{\varphi}$-precompact orbit. Let $\varepsilon>0$ and $x\in M_{0}$ be given, and $(g_{n})_{n}$ a sequence in $G$. There exists $y\in M_{00}$ such that $\norm{y\Omega_{\varphi}-x\Omega_{\varphi}}_{\varphi}<\varepsilon/3$. Now since $y$ has precompact orbit, there is a subsequence $(g_{n_{k}})$ such that $\beta_{g_{n_{k}}}(y)\Omega_{\varphi}$ is $\norm{\cdot}_{\varphi}$--Cauchy. Choose $N\in \mathbb{N}$ such that $k>l>N$ implies that $\norm{(\beta_{g_{n_{k}}}(y)-\beta_{g_{n_{l}}}(y))\Omega_{\varphi}}_{\varphi}<\varepsilon/3$. Then
\begin{align}
    \norm{(\beta_{g_{n_{k}}}(x)-\beta_{g_{n_{l}}}(x))\Omega_{\varphi}}_{\varphi}&\le \norm{(\beta_{g_{n_{k}}}(x)-\beta_{g_{n_{k}}}(y))\Omega_{\varphi}}_{\varphi}\\\nonumber&+\norm{(\beta_{g_{n_{k}}}(y)-\beta_{g_{n_{l}}}(y))\Omega_{\varphi}}_{\varphi}\\\nonumber
    &+\norm{(\beta_{g_{n_{l}}}(y)-\beta_{g_{n_{l}}}(x))\Omega_{\varphi}}_{\varphi}<\varepsilon,
\end{align}
and therefore $x\Omega_{\varphi}$ has precompact orbit. In fact, we have proved that $M_{0}\Omega_{\varphi}=\mathcal{H}_{c}\cap M\Omega_{\varphi}$.

We claim that $M_{0}=M_{K}$. 

To show $M_{K}\subseteq M_{0}$, note that if $\mathcal{K}$ is a finite-dimensional $\beta$-invariant subspace of $M^\varphi\Omega_\varphi$, then the orbit under $V$ of any $x \Omega_\varphi$ in $\mathcal{K}$ is $\|\cdot\|_{\varphi}$--precompact, and so $x\Omega_{\varphi}\in \mathcal{H}_{c}\cap M\Omega_{\varphi}\subseteq M_0\Omega_{\varphi}$ and therefore $x\in M_{0}$. Since $M_{K}$ is the unital von Neumann algebra generated by all $x$ taken from invariant finite dimensional subspaces $\mathcal{K}$, it follows that $M_{K}\subseteq M_{0}$.

To show $M_{0}\subseteq M_{K}$, assume $x\in M_{0}$. For every $\varepsilon>0$ there exists an $N\in \mathbb{N}$ so that $x\Omega_{\varphi}\in \mathcal{H}_{c}\cap M\Omega_{\varphi}$ is within $\varepsilon$ of $\bigoplus_{i=1}^{N}S_{i}\Omega_{\varphi}$ with each $S_{i}\subseteq M_{K}$ a finite-dimensional invariant subspace, by the above Peter-Weyl decomposition, an internal direct sum decomposition of $\mathcal{H}_{c}$. The fact that this is an internal direct sum is crucial, since the Hilbert space sum is compatible with the sum in the algebra, and hence there are $s_{i}\in S_{i}$, $i\in \{1,\ldots,N\}$ so that $\norm{(x-\sum_{i=1}^{N}s_{i})\Omega_{\varphi}}_{\varphi}<\varepsilon$ with $\sum_{i=1}^{N}s_{i}\in M_{K}$. Therefore  $x\Omega_{\varphi}\in [M_{K}\Omega_{\varphi}]$. Since $M_{K}\subset M^{\varphi}$ and $M_{0}\subset M^{\varphi}$ are inclusions of finite von Neumann algebras with traces (restrictions of) $\varphi$, the orthogonal projection of $[M^{\varphi}\Omega_{\varphi}]$ onto $[M_{K}\Omega_{\varphi}]$ restricts to the $\varphi$--preserving conditional expectation $\mathbb{E}=\mathbb{E}_{M_{K}}^{M^{\varphi}}$, and hence $\mathbb{E}(x)\Omega_{\varphi}=x\Omega_{\varphi}$ and since $\Omega_{\varphi}$ is separating for $M$, $\mathbb{E}(x)=x$ and consequently $x\in M_{K}$.

Finally, the injectivity of $M_K$ is a well-known result \cite{HLS}.\end{proof}

As a consequence of Theorem \ref{theorem:Kronecker}, we can make the following definition.  

\begin{definition}
The Kronecker subsystem $\mathfrak{M}_K$ of an ergodic system $\mathfrak{M}=(M,G,$ $\beta,\varphi)$ is the subsystem of $\mathfrak{M}$ whose underlying von Neumann algebra $M_{K}$ is generated by all finite-dimensional $\beta$-invariant subspaces of $M$.
\end{definition}

\noindent The proof of Theorem \ref{theorem:Kronecker} establishes that, in the above notation $\mathcal{H}_c =\overline{M_K \Omega_{\varphi}} $, that is, the Kronecker subsystem $\mathfrak{M}_K$ of an ergodic system $\mathfrak{M}$ is maximal among compact subsystems of $\mathfrak{M}$.  From this observation we obtain the following corollaries.

\begin{corollary}\label{weakmixingcorollary1}
Let $\mathfrak{M}$ be an ergodic system. For every finite subset $F$ of $M\ominus B_{0}$ then for every $\epsilon>0$ there exists a $g\in G$ with $|\varphi(x^{*}\beta_{g}(y))|<\epsilon$ for all $x,y\in F$.
\end{corollary}

\begin{corollary} \label{weakmixingcorollary2}
Let $\mathfrak{M}$ be an ergodic system with Kronecker subsystem $\mathfrak{M}_{K}$. If $y\in M\ominus M_{K}$ then $y$ is weakly mixing for $\beta$.
\end{corollary}
\begin{proof}
Let $\epsilon>0$ and $F$ be a finite subset of $M\ominus \mathbb{C}1$, and $F_{0}=\{x-\mathbb{E}_{M_{K}}(x):x\in F\}$. By Corollary \ref{weakmixingcorollary1} we know there exists $g\in G$ so that $|\varphi(x_{0}^{*}\beta_{g}(y))|<\epsilon$ for all $x_{0}\in F_{0}$. Now for any $x\in F$ we have that 
\begin{align*}
|\varphi(x^{*}\beta_{g}(y))|&=|\varphi((x-\mathbb{E}_{M_{K}}(x))^{*}\beta_{g}(y))+\varphi(\mathbb{E}_{M_{K}}(x)^{*}\beta_{g}(y)|\\
&=|\varphi((x-\mathbb{E}_{M_{K}}(x))^{*}\beta_{g}(y))|<\epsilon.
\end{align*} \end{proof}

\noindent The following additional corollary to Theorem \ref{theorem:Kronecker} extends Theorem 5.5 of \cite{Du} from the setting of abelian groups to that of a general locally compact acting group.  

\begin{corollary} \label{weakmixingcorollary3} An ergodic system $\mathfrak{M} = (M,G,\beta,\varphi)$ is weakly mixing if and only if its Kronecker subsystem $\mathfrak{M}_K$ is trivial.
\end{corollary}
\begin{proof}
If $\mathfrak{M}_K$ is the trivial system, then by \ref{weakmixingcorollary2}  every $y \in M \ominus \mathbb{C}1$ is weakly mixing for $\beta$, that is, the system $\mathfrak{M}$ is weakly mixing.  Conversely, if $\mathfrak{M}$ is weakly mixing and $\mathfrak{M}_K$ is not trivial, then the inclusion of the Kronecker subalgebra $M_K$ into $M$ defines a nontrivial compact subsystem of $\mathfrak{M}$, violating Theorem 6.13 of \cite{BCM}.  Thus, $\mathfrak{M}_K$ is trivial.  \end{proof}

The following is  our main result, which extends Theorem 5.6 of \cite{Du2} to actions of nonabelian groups.

\begin{theorem} Let $\mathfrak{N}$ be a compact subsystem of an ergodic system $\mathfrak{M} = (M,G,$ $\beta,\varphi)$.  If $\mathfrak{M}$ and $\mathfrak{M}_K$ are disjoint relative to $\mathfrak{N}$, then the subsystems $\mathfrak{N}$ and $\mathfrak{M}_K$ are isomorphic.   On the other hand, if $\mathfrak{N}$ is isomorphic to $\mathfrak{M}_K$ then for any compact system $\mathfrak{B}$ which has $\mathfrak{M}_K$ as a subsystem, $\mathfrak{M}$ and $\mathfrak{B}$ are disjoint relative to $\mathfrak{M}_K$.  
\end{theorem}

\begin{proof}  If $\mathfrak{N} = (N,\iota_M)$ is a compact subsystem of the ergodic system $\mathfrak{M}$ which is not isomorphic to $\mathfrak{M}_K$ then $N$ embeds in $M$ as a proper von Neumann subalgebra of $M_K$.  Denote the embedding of $\mathfrak{N}$ into $\mathfrak{M}_K$ by $\iota_{M_K}$.  Then, composition with the conditional expectation $E_{M_K}$ of $M$ onto $M_K$  defines a map $E_{M_K} \circ \iota_M^*$ in $J_\mathfrak{N}(\mathfrak{M},\mathfrak{M}_K)$ which is not equal to $\iota_{M_K}$.  This proves the first statement. 

For the second statement, let $\mathfrak{N}$ be a compact subsystem of $\mathfrak{M}$ isomorphic to $\mathfrak{M}_K$.  Then, up to a twist by the isomorphism which intertwines the two systems,  we may assume that $\mathfrak{N}=\mathfrak{M}_K$.  Let  $\mathfrak{B}=(B,G,\alpha,\rho)$ be a compact system as above, and write $\iota_B$ for the embedding of $\mathfrak{M}_K$ into $\mathfrak{B}$.  If $\Phi: B \rightarrow M$ is any joining of the two systems over $\mathfrak{M}_K$, then for any $x \in B$, $\Phi(x) \in M$ has precompact orbit, since $$ \beta_G(\Phi(x)) = \sett{\beta_g(\Phi(x)):g \in G} = \sett{\Phi(\alpha_g(x)):g \in G} = \Phi(\alpha_G(x)).$$  The maximality property of $\mathfrak{M}_K$ then implies $\Phi(x) \in M_K$.  Then, by Theorem \ref{Projection Formula}, for any $x \in B$ we have
$$ P(x \xi_\Phi) = \xi_\Phi \sigma_{i/2}^\phi(\Phi(x)) \in \Omega_\phi M_K,$$
and by Theorem \ref{relative independence characterization} we have that $\Phi$ is the relatively independent joining. \end{proof}




\end{document}